\documentclass[reqno]{amsart}
\usepackage{hyperref}

\begin{document}
\title[Ill-posedness for KdV-Burgers equation]
{ Remark on the ill-posedness for KdV-Burgers equation in Fourier amalgam spaces}

\author[D. G. Bhimani, S. Haque\hfil \hfilneg] {Divyang G. Bhimani, Saikatul Haque}  

\address{D. G. Bhimani \newline
Department of Mathematics, Indian Institute of Science Education and Research, Dr. Homi Bhabha Road, Pune 411008, India}
\email{divyang.bhimani@iiserpune.ac.in}

\address{S. Haque \newline
Harish-Chandra Research Institute, Chhatnag Road, Jhunsi, 
Prayagraj (Allahabad) 211019, India}
\email{saikatulhaque@hri.res.in}

\begin{abstract}
We have established (a weak form of) ill-posedness for the KdV-Burgers equation on a real line in Fourier amalgam spaces $\widehat{w}_s^{p,q}$ with $s<-1$. The particular case $p=q=2$ recovers the result in \cite{molinet2002low}. The result is new even in Fourier Lebesgue space $\mathcal{F}L_s^q$ which corresponds to the case $p=q(\neq 2)$ and in modulation space $M_s^{2,q}$ which corresponds to the case $p=2,q\neq 2$.
\end{abstract}

\maketitle \numberwithin{equation}{section}
\newtheorem{theorem}{Theorem}[section]
\newtheorem{corollary}[theorem]{Corollary}
\newtheorem{lemma}[theorem]{Lemma}
\newtheorem{remark}[theorem]{Remark}
\newtheorem{problem}[theorem]{Problem}
\newtheorem{example}[theorem]{Example}
\newtheorem{definition}[theorem]{Definition}
\allowdisplaybreaks

\section{Introduction}
In this paper we consider  the Korteweg-de Vries--Burgers (KdV-B) equation posed on the real line:
\begin{equation}\label{KdV-B}
\begin{cases}
u_t + u_{xxx}-u_{xx}+uu_{x}=0\\
u(0, x)=u_0
\end{cases} (t, x)\in \mathbb R \times \mathbb R,
\end{equation}
where $u=u(t,x) \in \mathbb R.$

The  KdV-B \eqref{KdV-B}  was derived as a model for the propagation of long, weakly nonlinear dispersive waves in certain physical contexts when dissipative effects occur (see \cite{van1970}).  Many authors have studied  the short and long time behaviours of solutions of   KdV,  KdV-B  and several of its variants in the context of Sobolev spaces,  see e.g.   \cite{molinet2002low,  molinet2011sharp} and the references therein.  In recent years, there has been a great deal of interest in studying dispersive PDEs  with Cauchy data in low regularity spaces.  See e.g.  survey article \cite{mr}.  Recently,  in \cite{oh} authors  have studied KdV in modulation spaces, and in \cite{bhimani2021strong, bhimani2021norm, bhimani2022norm} authors have studied ill-poedesness for wave, BBM and NLS in Fourier amalgam spaces, see also \cite{oh2017remark,kishimoto2019remark}. 
We refer to \cite{bejenaru2006sharp} by Bejenaru-Tao for abstract well-posedness and ill-posedness theory.
However,  we note that there are no known well/ill-posedness  results  for \eqref{KdV-B} in modulation and Fourier amalgam spaces.

In this note,  we would like to  initiate the study of ill-posedness for \eqref{KdV-B} in the realm of Fourier amalgam spaces.  We now briefly recall these spaces.  In order to study well-posednes for 1D cubic nonlinear Schr\"odinger equations, in \cite{jfth, jfto},  Oh and  Forlano have  introduced the Fourier amalgam space $\widehat{w}_s^{p,q} \  (1\leq p, q \leq \infty,  s\in \mathbb R):$
$$\widehat{w}^{p,q}_s(\mathbb R)=\left \{ f\in \mathcal{S}'(\mathbb R): \|f\|_{\widehat{w}^{p,q}_s}=   \left\| \left\lVert \chi_{n+Q}(\xi)\mathcal F f(\xi)\right\rVert_{L_\xi^p(\mathbb R)} \langle n \rangle^s \right\|_{\ell^q_n(\mathbb Z)}< \infty \right\}
$$where $Q=(-1/2, 1/2]$ and $\mathcal{F}$ denote the Fourier transform.
These spaces recapture several known spaces: 
\begin{equation*} 
   \widehat{w}_s^{p,q}(\mathbb R)= \begin{cases}
   \mathcal{F}L_s^q(\mathbb R) \ \text{(Fourier-Lebesgue space)} \quad \textit{if}  \quad p=q\\
   M^{2,q}_s(\mathbb R) \ \text{(modulation space)}  \quad \textit{if}  \quad p=2\\
   H^{s}(\mathbb R) \  \text{(Sobolev space)}  \quad \textit{if} \quad p=q=2.
    \end{cases}
    \end{equation*}
    See also Remarks \ref{fw0}, \ref{fw}, \ref{fw2}.  We now state our main result.
\begin{theorem}\label{mt}
Let $s < -1$. Then there does not exist any $T > 0$ such that \eqref{KdV-B} admits a
unique local solution defined on the interval $[0, T ]$ and such that the flow-map $u_0\mapsto u(t),\ t\in[0,T]$
is $C^2$-differentiable at zero from $\widehat{w}_s^{p,q}(\mathbb R)$ to $C([0,T];\widehat{w}_s^{p,q}(\mathbb R))$.
\end{theorem}

Theorem \ref{mt} is new  for $p=q\neq 2$ and recovers the result of Molinet-Ribaud in \cite[Theorem 1.2]{molinet2002low}.   The method of proof for Theorem \ref{mt}  rely on showing the  ``unboundednes" of the second Picard iterate associated with  \ref{KdV-B}- which was initiated by Bourgain \cite{b1}  for KdV and later for mKdV by N. Tzvetkov in \cite{tz}.  See also \cite[Section 4]{bhimani2023hartree} and the references there in   for the further  comments.

We plan to address the norm-inflation (the stronger phenomenon than the mere ill-poseness) and even the worst situation  of norm inflation with infinite loss of regularity for \eqref{KdV-B} in our future works.  We note that recently  similar questions we have already addressed for Hartree,  nonlinear Schr\"odinger,  BBM and wave equations in \cite{bhimani2021strong, bhimani2021norm, bhimani2022norm}.   We also expect to develop well-posedness theory for \eqref{KdV-B} in $\widehat{w}^{p,q}$ space in the future.

We conclude this section by following  remarks. 
\begin{remark}\label{fw0}
Recall that the Fourier Lebesgue space $\mathcal{F}L_s^q$ is defined by $\{f\in\mathcal{S}'(\mathbb{R}^d):\|f\|_{\mathcal{F}L_s^q}:=\|\langle\xi\rangle^s\mathcal{F}f\|_{L^q}<\infty\}$.
\end{remark}
 
\begin{remark}\label{fw} For any given function $f$ which is locally in $B$  (Banach space) (i.e,  $gf\in B, \forall g \in C_0^{\infty}(\mathbb R^d)),$ we set $f_{B}(x)= \|f g(\cdot -x)\|_{B}.$ The Feichtinger's \cite{}  Wiener amalgam  space $W(B, C)$ endowed with the norm  $\|f\|_{W(B, C)}=\|f_{B}\|_{C}.$  The Fourier amalgam spaces is a Fourier image of particular Wiener amalgam spaces,  specifically,  $\mathcal{F}W(L^p, \ell^q_s)=\widehat{w}^{p,q}_s.$
\end{remark}

\begin{remark}\label{fw2} Let   $\rho \in \mathcal{S}(\mathbb R^d),$  $\rho: \mathbb R^d \to [0,1]$  be  a smooth function satisfying   $\rho(\xi)= 1$ if $|\xi|_{\infty}=\max(|\xi_1|,...,|\xi_d|)\leq \frac{1}{2} $ and $\rho(\xi)=0$ if $|\xi|_{\infty}\geq  1.$ Let  $\rho_n$ be a translation of $\rho,$ that is,
$ \rho_n(\xi)= \rho(\xi -n), n \in \mathbb Z^d$
and denote 
$\sigma_{n}(\xi)=
  \frac{\rho_{n}(\xi)}{\sum_{\ell\in\mathbb Z^{d}}\rho_{\ell}(\xi)},  n
  \in \mathbb Z^d.$ 
  Then the frequency-uniform decomposition operators can be defined by 
\[\square_n = \mathcal{F}^{-1} \sigma_n \mathcal{F},\qquad n\in\mathbb{Z}^d. \]
The modulation $M^{p,q}_s(\mathbb R^d)$ 
(with $1\leq p, q \leq \infty, s\in \mathbb R$) is defined by the norm:  
\begin{equation*}
M^{p,q}_s(\mathbb R^d)=\left\{f\in\mathcal{S}'(\mathbb{R}^d):\|f\|_{M^{p,q}_s(\mathbb R)}:=   \left\| \left\lVert
  \square_nf\right\rVert_{L_x^p(\mathbb R)} \langle n \rangle ^{s} \right\|_{\ell^q_n(\mathbb Z^d)} \right\}. 
\end{equation*}
\end{remark}

\section{Proof of Theorem \ref{mt}}
The integral version of \eqref{KdV-B} is given by  
\begin{equation}\label{inteq}
u(t)=S(t)u_0-\frac{1}{2}\int_0^tS(t-\tau)\partial_x[u(\tau)]^2d\tau
\end{equation}
where $\{S(t)\}_{t\geq0}$ given by
\begin{equation}\label{sg}
\mathcal F S(t)u_0=e^{-t\xi^2+it\xi^3}\mathcal F u_0,\qquad t\geq0
\end{equation}
 is the semi-group associated to the linear part of \eqref{KdV-B}. The proof of Theorem \ref{mt} follows from the fact that the second Picard  iterate (given by \eqref{2ndit}) associated to \eqref{inteq} is not continuous  at zero from $\widehat{w}_s^{p,q}(\mathbb R)$ to $C([0,T];\widehat{w}_s^{p,q}(\mathbb R))$. We refer to the proof of Theorem 1.10 in \cite{bhimani2023hartree} and the reference therein for detail. 

\begin{proof}[Proof of Theorem \ref{mt}]
We define the sequence of initial data $\{\phi_N\}_{N\geq 1}$ by
\begin{equation*}
\mathcal F{\phi_N} = N \left( \chi_{I_N} + \chi_{I_N}(-\xi) \right)
\end{equation*}
where $I_N=[N, N+2]$ and $\mathcal F{\phi_N}$ denotes the space Fourier transform of $\phi_N.$ 

Let us compute $ \|\phi_{N}\|_{\widehat{w}^{p,q}_s}$.
Note that with $\Omega=I_N\cup (-I_N)$
\begin{align*}
 \|\phi_{N}\|_{\widehat{w}^{p,q}_s}=   \left\| \left\lVert \chi_{n+Q}(\xi)\mathcal F \phi_N\right\rVert_{L_\xi^p(\mathbb R)}  \langle n \rangle^s \right\|_{\ell^q_n(\mathbb Z)}=  N \left\| \left\lVert \chi_{n+Q}(\xi)\chi_\Omega(\xi)\right\rVert_{L_\xi^p(\mathbb R)}  \langle n \rangle^s \right\|_{\ell^q_n(\mathbb Z)}
\end{align*}
Now $\|\chi_{n+Q}(\xi)\chi_\Omega(\xi)\|_{L_\xi^p(\mathbb R)}$ survives only if $n\in \mathcal{G}:=\{m\in\mathbb Z:(m+Q)\cap\Omega\neq\emptyset\},$ and for these $n$'s one must have $|n|\sim N$. Since $\#(\mathcal{G})\sim 1$ and $\|\chi_{n+Q}(\xi)\chi_\Omega(\xi)\|_{L_\xi^p(\mathbb R)}=\|\chi_{n+Q}(\xi)\|_{L_\xi^p(\mathbb R)}=1$ for almost all $n\in\mathcal{G}$,  we conclude
\begin{eqnarray*}
 \|\phi_{0,N}\|_{\widehat{w}^{p,q}_s}&=&N \left\| \left\lVert \chi_{n+Q}(\xi)\chi_\Omega(\xi)\right\rVert_{L_\xi^p(\mathbb R)}  \langle n \rangle^s \right\|_{\ell^q_n(\mathbb Z)}\\
 &=& N\left(\sum_{n\in\mathcal{G}}\|\chi_{n+Q}(\xi)\chi_\Omega(\xi)\|_{L_\xi^p(\mathbb R)}^q\langle n \rangle^{sq}\right)^{1/q}\\
 &\sim& N(N^{sq})^{1/q}=N^{1+s}.
\end{eqnarray*}

Therefore
 $\|\phi_N\|_{\widehat{w}^{p,q}_{-1}}\sim 1$ and $\phi_N \to 0$ in $\widehat{w}^{p,q}_s(\mathbb R)$ for $s<-1.$ This sequence yields  a counter example to the continuity of the second iteration of the Picard Scheme in $\widehat{w}^{p,q}_s(\mathbb R)$ for $s<-1$ that is given by 
\begin{eqnarray}\label{2ndit}
A_2(t, h, h) = \int_0^t S(t-\tau) \partial_x[S(\tau)h]^2 d\tau
\end{eqnarray}
where $\{S(t)\}_{t\geq0}$ is defined in \eqref{sg}.  Indeed,  computing the space Fourier transform we get

\begin{eqnarray*}
 && \mathcal{F} (A_2(t, \phi_N, \phi_N))(\xi)\\
 &=&\int_0^te^{-(t-\tau)\xi^2+i(t-\tau)\xi^3}(i\xi)[\mathcal F S(\tau)\phi_N\ast\mathcal F S(\tau)\phi_N](\xi)d\tau\\
 &=&\int_0^te^{-(t-\tau)\xi^2+i(t-\tau)\xi^3}(i\xi)\int_{\mathbb R}[\mathcal F S(\tau)\phi_N](\xi-\xi_1)[\mathcal F S(\tau)\phi_N](\xi_1)d\xi_1d\tau\\
 &=&\int_0^te^{-(t-\tau)\xi^2+i(t-\tau)\xi^3}(i\xi)\int_{\mathbb R}e^{-\tau(\xi-\xi_1)^2+i\tau(\xi_xi_1)^3}\mathcal F \phi_N(\xi-\xi_1)\\
 &&\hspace{5cm}\times e^{-\tau\xi_1^2+i\tau\xi_1^3}\mathcal F \phi_N(\xi_1)d\xi_1d\tau\\
  &  =&  \int_{\mathbb R} e^{-t\xi^2} e^{it \xi^3} \mathcal F{\phi_N}(\xi_1) \mathcal F{\phi_N} (\xi-\xi_1)
(i \xi)\\
 &&\hspace{4cm}\times \int_0^t e^{- (\xi_1^2 + (\xi- \xi_1)^2 - \xi^2)\tau} e^{i (\xi_1^3 + (\xi-\xi_1)^3- \xi^3) \tau} d\tau d\xi_1\\
& =&  e^{-t\xi^2} e^{it \xi^3} (i \xi) \int_{\mathbb R}  \mathcal F{\phi_N}(\xi_1) \mathcal F{\phi_N} (\xi-\xi_1)
 \frac{e^{-(\xi_1^2 + (\xi-\xi_1)^2-\xi^2)t} e^{i3 \xi \xi_1 (\xi-\xi_1)t} - 1}{-2\xi_1 (\xi-\xi_1) + i3 \xi \xi_1 (\xi-\xi_1)}  d\xi_1.
\end{eqnarray*} 
We note that 
\begin{eqnarray}\label{1}
&&\left |\mathcal{F} (A_2(t, \phi_N, \phi_N))(\xi) \right|\nonumber\\
& = &N^{2}|\xi| \left| \int_{K_{\xi}} \frac{e^{-(\xi_1^2 + (\xi-\xi_1)^2)t} e^{i3 \xi \xi_1 (\xi-\xi_1)t} - e^{-\xi^2 t}}{-2\xi_1 (\xi-\xi_1) + i3 \xi \xi_1 (\xi-\xi_1)} d\xi_1 \right|
\end{eqnarray}
where
\[K_{\xi}=\{ \xi_1: \xi-\xi_1 \in I_N,  \xi_1\in -I_N \} \cup \{ \xi_1:  \xi_1 \in I_N,  \xi- \xi_1 \in -I_N \}. \]
Note that for any $\xi \in [-1/2, 1/2],$ one has $|K_{\xi}| \geq 1$ and 
\begin{equation*}
\begin{cases} 
3 \xi \xi_1 (\xi- \xi_1) \leq cN^2\\
2 \xi_1(\xi-\xi_1) \sim N^2
\end{cases} \forall \xi_1 \in K_\xi.
\end{equation*}
Therefore,  fixing $0<t<1,$ for $\xi\in[-1/2, 1/2],$ we have 
\[ \text{Re} \left( e^{-( \xi_1^2 + (\xi- \xi_1)^2)t)}  e^{i3\xi  \xi_1 (\xi- \xi_1)t} - e^{- \xi^2 t} \right) \leq -e^{-t/4} + e^{-2(N+2)^2 t} \]
which leads for $N=N(t)>0$ large enough (so that $e^{-2(N+2)^2 t}\leq \frac{1}{2}e^{-t/4}$) to 
\begin{eqnarray}\label{2}
\left| \int_{K_{\xi}} \frac{e^{-(\xi_1^2 + (\xi-\xi_1)^2)t} e^{i3 \xi \xi_1 (\xi-\xi_1)t} - e^{-\xi^2 t}}{-2\xi_1 (\xi-\xi_1) + i3 \xi \xi_1 (\xi-\xi_1)} d\xi_1 \right| &\geq& c \frac{e^{-t/4}}{N^2}. 
\end{eqnarray}
Now using the fact $\|a_n\|_{\ell_n^q(\mathbb Z)}\geq a_0$ we have
\begin{eqnarray*}
\|A_2(t, \phi_N, \phi_N)\|_{\widehat{w}^{p,q}_s}&=& \left\| \left\lVert \chi_{n+Q}(\xi)\mathcal F A_2(t, \phi_N, \phi_N)(\xi)\right\rVert_{L_\xi^p(\mathbb R)} \langle n \rangle^s \right\|_{\ell^q_n(\mathbb Z)}\\
&\geq&\left\lVert \chi_{0+Q}(\xi)\mathcal F A_2(t, \phi_N, \phi_N)(\xi)\right\rVert_{L_\xi^p(\mathbb R)} \\
&=&\left(\int_{(-\frac{1}{2},\frac{1}{2}]}|\mathcal F A_2(t, \phi_N, \phi_N)(\xi)|^pd\xi\right)^{1/p}.
\end{eqnarray*}
Using \eqref{1}, \eqref{2} we have
\[
\|A_2(t, \phi_N, \phi_N)\|_{\widehat{w}^{p,q}_s}\geq cN^2\frac{e^{-t/4}}{N^2}\left(\int_{(-\frac{1}{2},\frac{1}{2}]}|\xi|^pd\xi\right)^{1/p}\geq c_0
\]
for some positive  constant $c_0$ independent of $N$ and $t$.

Since  $\phi_N \to 0$ in $\widehat{w}^{p,q}_s(\mathbb R),$ for $s<-1,$ this ensure that,  for any fixed $t>0,$ the map $u_0 \mapsto A_2(t, u_0, u_0)$ is not continuous at zero from $\widehat{w}^{p,q}_s(\mathbb R)$ into $\widehat{w}^{p,q}_s(\mathbb R).$
\end{proof}

\section*{Acknowledgements}
S Haque  is thankful to  DST-INSPIRE (DST/INSPIRE/04/2022/001457) for  financial support. S Haque is also thankful to Harish-Chandra Research Institute for the excellent research facilities. 


\begin{thebibliography}{99}
\bibitem{bejenaru2006sharp}
I. Bejenaru, T. Tao, Sharp well-posedness and ill-posedness results for a quadratic non-linear Schr\"odinger equation, Journal of functional analysis, 233 (2006), pp. 228--259.

\bibitem{bhimani2021strong}
D. G. Bhimani, S. Haque, Strong ill-posedness for fractional Hartree and cubic NLS Equations, arXiv:2101.03991, (2021).

\bibitem{bhimani2021norm}
D. G. Bhimani, S. Haque, Norm inflation for Benjamin–Bona–Mahony equation in Fourier amalgam and Wiener amalgam spaces with negative regularity, Mathematics, 9 (2021), 23, 3145.

\bibitem{bhimani2022norm}
D. G. Bhimani, S. Haque, Norm inflation with infinite loss of regularity at general initial data for nonlinear wave equations in Wiener amalgam and Fourier amalgam spaces, Nonlinear Analysis, 223 (2022), 113076

\bibitem{bhimani2023hartree}
D. G. Bhimani, S. Haque, The Hartree and Hartree-Fock Equations in Lebesgue $\widehat{L}^p$ and Fourier-Lebesgue $\widehat{L}^p$ spaces, Ann. Henri Poincaré 24 (2023), no. 3, 1005–1049.

\bibitem{van1970} D. R. Van Dooren,   Comments on: “Stability of solitary waves in shallow water” (Phys. Fluids 19 (6) (1976),   771-777.


\bibitem{b1} J.  Bourgain,   Periodic korteweg de vries equation with measures as initial data. Sel. Math. New Ser. 3,  (1997)  115-159.

\bibitem{molinet2002low}
L. Molinet and F. Ribaud, On the low regularity of the Korteweg-de Vries-Burgers equation, Int. Math. Res. Not., (2002), pp. 1979–2005.

\bibitem{molinet2011sharp}
 L. Molinet and S. Vento, Sharp ill-posedness and well-posedness results for the KdV-Burgers equa- tion: the real line case, Ann. Sc. Norm. Super. Pisa Cl. Sci. (5), 10 (2011), pp. 531–560.




\bibitem{fw} H. G. Feichtinger,  Banach convolution algebras of Wiener type,  in Functions, series, operators,  Vol. I, II
(Budapest, 1980), vol. 35 of Colloq. Math. Soc. J\'anos Bolyai, North-Holland, Amsterdam, 1983, pp. 509-524.

\bibitem{jfth} J. Forlano,  On the deterministic and probabilistic Cauchy problem of nonlinear dispersive partial differential  equations,    PhD thesis, Heriot-Watt University,  2020.

\bibitem{jfto} J. Forlano and T. Oh, Normal form approach to the one-dimensional cubic nonlinear Schr\"odinger equation
in Fourier-amalgam spaces, preprint.
\bibitem{kishimoto2019remark}
N. Kishimoto, A remark on norm inflation for nonlinear Schr\"odinger equations, Communications on Pure \&
Applied Analysis, 18 (2019), p. 1375.


\bibitem{mr} M. Ruzhansky,   M.   Sugimoto, 
B.  Wang,  Modulation spaces and nonlinear evolution equations.   Evolution equations of hyperbolic and Schr\"odinger type,  267–283, Progr. Math., 301, Birkh\"auser/Springer Basel AG, Basel, 2012.

\bibitem{tz} N. Tzvetkov,  Remark on the local ill-posedness for kdv equation. Comptes Ren- dus de l’Acad\'emie des Sciences-Series I-Mathematics 329, (1999) 1043-1047.

\bibitem{oh2017remark}
T. Oh, A remark on norm inflation with general initial data for the cubic nonlinear Schr\"odinger equations in
negative Sobolev spaces, Funkcialaj Ekvacioj, 60 (2017), pp. 259–277.

\bibitem{oh} T. Oh,  Y.  Wang,  On global well-posedness of the modified KdV equation in modulation spaces,   Discrete Contin.  Dyn.  Syst.  41 (6) (2021), 2971-2992.
\end{thebibliography}
\end{document}